\documentclass[a4paper,10pt]{article}
\usepackage{authblk}
\usepackage[usenames,dvipsnames,svgnames,table]{xcolor}
\usepackage{graphicx}
\usepackage{multirow}
\usepackage{amssymb}
\usepackage{epstopdf}
\usepackage{pst-node}
\usepackage{pstricks}
\usepackage{xcolor}
\usepackage{pgf}
\usepackage{graphicx}
\usepackage{pict2e}
\usepackage{color}
\usepackage{chessboard}
\usepackage{ragged2e}
\usepackage{booktabs}
\usepackage{algorithm}
\usepackage{algorithmic}
\usepackage{chessboard} 
\usepackage{subcaption} 

\usepackage{tikz}
\usepackage{tikz-3dplot}
\usetikzlibrary{arrows.meta}
\usetikzlibrary{shapes,arrows}
\usepackage{wrapfig} 

\usepackage[centertags]{amsmath}
\usepackage{amsfonts}
\usepackage{amsthm}
\usepackage{newlfont}
\usepackage{color}
\usepackage{float}
\newtheorem{theorem}{Theorem}[section]
\newtheorem{corollary}[theorem]{Corollary}

\newtheorem{definition}[theorem]{Definition}
\newtheorem{example}[theorem]{Example}
\newtheorem{remark}[theorem]{Remark}

\title{Sharp bounds on the least eigenvalue of a graph determined from edge clique partitions}

\author[1,2]{Domingos M. Cardoso}
\author[1,2]{In\^es Ser\^odio Costa}
\author[1,2]{Rui Duarte}

\affil[1]{\small Centro de Investiga\c{c}\~{a}o e Desenvolvimento em Matem\'atica e Aplica\c{c}\~{o}es}
\affil[2]{\small Departamento de Matem\'atica, Universidade de Aveiro, 3810-193, Aveiro, Portugal.}

\begin{document}

\maketitle

\begin{abstract}
Sharp bounds on the least eigenvalue of an arbitrary graph are presented. Necessary and sufficient (just sufficient) conditions for the lower (upper) bound to be attained are deduced using edge clique partitions. As an application, we prove that the least eigenvalue of the $n$-Queens' graph $\mathcal{Q}(n)$ is equal to $-4$ for every $n \ge 4$ and it is also proven that the multiplicity of this eigenvalue is $(n-3)^2$. Additionally, some results on the edge clique partition graph parameters are obtained.
\end{abstract}

\medskip


\smallskip


\section{Introduction}\label{sec1}

There are many more published bounds on the largest eigenvalue of a graph than bounds on the least eigenvalue.
The book published in 2015 by Dragan Stevanovi\'c \cite{Stevanovic2015} provides an overview of the
developments on the largest eigenvalue of a graph obtained in the 10 years prior to its publication. For the
least eigenvalue of a graph only some bounds are known and most of them are only achieved in very particular cases.
Chapter 3 of the book published in 2015 by Zoran Stani\'c \cite{Stanic2015} is entirely devoted to inequalities for
the least eigenvalue of a graph. Recently, a lower bound on the least eigenvalue and a necessary and sufficient condition
for this lower bound to be attained was deduced in \cite[Corollary 3.1]{CioabaElzingaGregory2020} for scalar multiples
of graphs, using an approach motivated by results on line graphs and generalized line graphs. In this paper, regarding
the lower bound, a similar result motivated by edge clique partitions \cite{Orlin1977}, independently obtained in arXiv-preprint (2020) \cite[Theorem 3.3]{CardosoCostaDuarte2020}, is presented. A related result in the particular
context of geometric distance-regular graphs appears in \cite[Proposition 9.8]{DamKoolenTanaka2016}.

As an application, we consider the Queens' graph $\mathcal{Q}(n)$, which is obtained from the $n\times n$ chessboard
where its squares are the vertices of the graph and two of them are adjacent if and only if they are in the same row,
column or diagonal of the chessboard. It will be proved that the least eigenvalue of $\mathcal{Q}(n)$ is equal to $-4$
for every $n \ge 4$ and its multiplicity is $(n-3)^2$.

Throughout the text we just consider simple undirected graphs. The vertex set of a graph $G$ is denoted by $V(G)$ and its edge set by $E(G)$. An edge with end-vertices $i$ and $j$ is denoted by $ij$. If $E' \subseteq E(G)$, then $G[E']$ denotes the subgraph of $G$ induced by the end-vertices of the edges in $E'$. The maximum degree of the vertices in $G$ is denoted by $\Delta(G)$. A vertex subset where each pair of vertices are (aren't) the end-vertices of an edge is called a clique (stable set) of $G$ and the maximum number of vertices forming a clique (stable set) in $G$ is the clique (stability) number of $G$. A stable set of maximum cardinality is called a maximum stable set. The adjacency matrix of a graph $G$ is denoted $A(G)$ and its eigenvalues are also called the eigenvalues of $G$. The spectrum of $G$, i.e. the multiset of eigenvalues is denoted by $\sigma(G)$. If $\mu$ is an eigenvalue of a graph $G$, the eigenspace associated to $\mu$ is denoted by $\mathcal{E}_G(\mu)$.\\

This paper is organized as follows. In the next section, we recall some useful edge clique partition graph
parameters. Furthermore, some results on these parameters are obtained and families of graphs with a particular edge clique partition property are presented. Section~\ref{sec3} includes the main results of this paper. Bounds on the least eigenvalue of a graph are deduced and necessary and sufficient (just sufficient) conditions for which the lower (upper) bound is attained are proven. Section~\ref{sec4} is devoted to the application of the main results to the Queens' graph $\mathcal{Q}(n)$. From this application, we conclude that the lower bound on the least eigenvalue of $\mathcal{Q}(n)$ is constant and attained for $n \ge 4$. This paper finishes with some conclusions and remarks in
Section~\ref{sec5}.

\section{Edge clique partitions}\label{sec2}

Edge clique partitions (ECP for short) were introduced in \cite{Orlin1977}, where the \textit{content} of a graph $G$,
denoted by $C(G)$, was defined as the minimum number of edge disjoint cliques whose union includes all the edges of $G$. Such minimum ECP is called in \cite{Orlin1977} \textit{content decomposition} of $G$. As proved in \cite{Orlin1977}, in general, the determination of $C(G)$ is $\mathbf{NP}$-Complete. Recently,
in \cite[Corollary 3.2]{ZhouDam2021}, a sharp lower bound on the content of a graph in terms of its largest eigenvalue,
minimum degree and clique number is deduced.

\begin{definition}(Clique degree and maximum clique degree)\label{new_graph_parameters}
	Consider a graph $G$ and an ECP, $P=\{ E_i \mid i \in I \}$. Then $V_i=V(G[E_i])$ is a clique of $G$ for every $i \in I$. For any $v \in V(G)$, the clique degree of $v$ relative to $P$, denoted $m_v(P)$, is the number of cliques $V_i$ containing the vertex $v$, and the maximum clique degree of $G$ relative to P, denoted $m_G(P)$, is the maximum of clique degrees of the vertices of $G$ relative to $P$.
\end{definition}

From Definition~\ref{new_graph_parameters}, considering an ECP, $P=\{ E_i \mid i \in I \}$, the parameters $m_v(P)$ and $m_G(P)$ can be expressed as follows.

\begin{eqnarray}
	m_v(P) &=& \{i\in I \mid v\in V(G[E_i])\},\qquad \forall v \in V(G);\label{parameter_m}\\
	m_G(P)&=& \max\{ m_v \mid v \in V(G) \}. \label{parameter_mv}
\end{eqnarray}

\begin{remark}\label{content_numbe_lower_bound_remark}
	It is clear that if $P$ is an ECP of $G$, then $m_G(P)$ is not greater than $|P|$. In particular, if $P$ is a content decomposition of $G$, then $m_G(P) \le C(G)$.
\end{remark}

\begin{example}\label{ex1}
	The Figure~\ref{figura_1} depicts a graph $G$ such that $V(G)=\{1, 2, 3, 4, 5\}$ and the ECP, $P=\{\{12, 23, 31\}, \{34, 45, 53\}, \{24\}\}$, which is a content decomposition of $G$. From Definition~\ref{new_graph_parameters} it follows that $m_v(P)=2$, if $v \in \{2, 3, 4\}$ and $m_v(P)=1$, if $v \in \{1, 5\}$. Therefore, $m_G(P)=2$.
\end{example}

\begin{figure}[ht]
	\centering
	\centering
	
	\begin{tikzpicture}
		\begin{scope}[every node/.style={circle,thick,draw}]
			\node (1) at (-1,1) {1};
			\node (2) at (-1,-1) {2};
			\node (3) at (0,0) {3};
			\node (4) at (1,-1) {4};
			\node (5) at (1,1) {5};
			\node (1') at (4,1) {1};
			\node (2') at (4,-1) {2};
			\node (3') at (5,0) {3};
			\node (4') at (6,-1) {4};
			\node (5') at (6,1) {5};
		\end{scope}
		\draw (0,-1.5) node {$G$};
		
		\begin{scope}[>={Stealth[black]},
			every edge/.style={draw=black,line width=0.5pt}]
			\path (1) edge node {} (2);
			\path (1) edge node {} (3);
			\path (2) edge node {} (3);
			\path (2) edge node {} (4);
			\path (3) edge node {} (4);
			\path (3) edge node {} (5);
			\path (4) edge node {} (5);
			\path (1') edge[color=red] node {} (2');
			\path (1') edge[color=red] node {} (3');
			\path (2') edge[color=red] node {} (3');
			\path (2') edge[color=cyan] node {} (4');
			\path (3') edge[color=green] node {} (4');
			\path (3') edge[color=green] node {} (5');
			\path (4') edge[color=green] node {} (5');
		\end{scope}
		\draw (4.65,0.55) node {\color{red}$a$};
		\draw (3.85,0) node {\color{red}$a$};
		\draw (4.65,-0.6) node {\color{red}$a$};
		\draw (5,-1.15) node {\color{cyan}$c$};
		\draw (5.35,0.55) node {\color{green}$b$};
		\draw (6.15,0) node {\color{green}$b$};
		\draw (5.35,-0.6) node {\color{green}$b$};
		
	\end{tikzpicture}
	
	\caption{A graph $G$ with a content decomposition where, on the right, the edges with the same color, among the colors {\color{red}$a$}, {\color{green}$b$} and {\color{cyan} $c$}, belong to the same part.} \label{figura_1}
\end{figure}
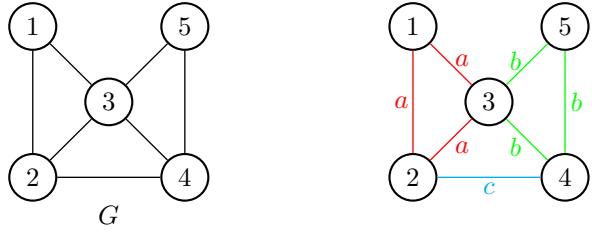

\begin{remark}
	The complete graphs $K_n$ are the unique connected graphs that admit the trivial ECP, $P=\{E(K_n)\}$, and thus
	$m_{K_n}(P)=1$. It is also immediate that if $P$ is an ECP of a graph $G$, then $m_{G}(P) \le \Delta(G)$. In the particular case of a tree, $T$, since each part of its unique ECP, $P=\{\{e\} \mid e \in E(T)\}$, is singleton, $m_{T}(P)= \Delta (T)$.
\end{remark}

The next theorem allows the construction of families of connected graphs $\mathcal{G}(H)=\{G_k \mid k \ge m_{H}(P)\}$, obtained from an arbitrary connected graph $H$ with an ECP, $P$, where each graph $G_k \in \mathcal{G}(H)$ has $H$ as a subgraph and admits an ECP, $P_k$, such that $m_{G_k}(P_k)=k$.

\begin{theorem}\label{family_of_graphs}
	Let $H$ be a connected graph with an ECP, $P$. Then for every $k \ge m_{H}(P)$ there exists a connected graph $G_k$
	which has $H$ as a subgraph and admits an ECP, $P_k$, such that $m_{G_k}(P_k)=k$.
\end{theorem}

\begin{proof}
	Consider a connected graph $H$ with an ECP, $P$, and a family of graphs $\mathcal{G}(H)=\{G_k \mid k \ge m_{H}(P)\}$,
	where $G_{m_{H}(P)}=H$ and for $k \ge m_H(P)$, each graph $G_{k+1}$ is obtained from $G_{k}$ as follows.
	Produce a copy, $G'_k$, of $G_k$ and consider a permutation $\pi_k$ on the vertices of $G_k$. Connect by an edge each vertex $v$ in $G_k$ to $\pi_k(v)=v' \in V(G'_k)$. Then, by construction, the graphs in $\mathcal{G}(H)$ are
	connected having $H$ as a subgraph. Furthermore, they are such that $V(G_{k+1}) = V(G_k) \cup V(G'_k)$ and
	$E(G_{k+1})=E(G_k) \cup E(G'_k) \cup M_k$, where $M_k=\{vv' \mid v \in V(G_k) \; \wedge \; v' \in V(G'_k)\}$, that is,
	$M_k$ is the matching corresponding to the assignment of the vertices in $G_k$ to their images by $\pi_k$ on its
	copies in $G'_k$. Assuming that $P_k$ is an ECP of $G_k$, for which $m_{G_k}(P_k)=k$ and $P'_k$ is the corresponding
	ECP of $G'_k$, then $P_{k+1} = P_k \cup P'_k \cup \{ \{e\} \mid e \in M_k\}$ is an ECP of $G_{k+1}$ for which $m_{G_{k+1}}(P_{k+1})=k+1$. Therefore, by induction on $k$, it follows that for every $k \ge m_{H}(P)$,
	$m_{G_k}(P_k)=k$.
\end{proof}

The above defined family of graphs
\begin{equation}
	\mathcal{G}(H) = \{G_k \mid k \ge m_{H}(P)\} \label{family_of graphs}
\end{equation}
depends from the initial graph $G_{m_{H}(P)}=H$ and from the permutations $\pi_k$. If the chosen graph $H$ admits an ECP, $P$, which is a content decomposition, as it is the case of the graph $G$ in Example~\ref{ex1}, it is immediate that for every $k \ge m_{H}(P)$, independently of the chosen permutations $\pi_k$, $P_k$ is a content decomposition of $G_k$. So this property is invariant to the permutations $\pi_k$.

\section{Main results}\label{sec3}

Using the above defined graph parameters, the next theorem states a lower bound on the least eigenvalue of a graph and a necessary and sufficient condition for to be attained in a particular ECP.

\begin{theorem}\label{Main_result_1}
	Let $P=\{ E_i \mid i \in I \}$ be an ECP of a graph $G$, $m=m_G(P)$ and $m_v=m_v(P)$ for every $v \in V(G)$. Then
	\begin{enumerate}
		\item If $\mu$ is an eigenvalue of $G$, then $\mu \ge -m$. \label{marca}
		\item $-m$ is an eigenvalue of $G$ if and only if there exists a vector $X \ne \mathbf{0}$ such that
		\begin{enumerate}
			\item $\sum \limits_{j \in V(G[E_i])} x_j = 0$, for every $i \in I$ and \label{cond1}
			\item $\forall v \in V(G) \;\; x_v=0$ whenever $m_v \ne m$. \label{cond2}
		\end{enumerate}\label{marca2}
		In the positive case, $X$ is an eigenvector associated with the eigenvalue $-m$.
	\end{enumerate}
\end{theorem}

\begin{proof}
	Let $A(G)$ be the adjacency matrix of $G$.
	\begin{enumerate}
		\item Let $X$ be an eigenvector of $A(G)$ associated with an eigenvalue $\mu$. Then
		\begin{align*}
			(\mu+m) \| X \|^2 & = X^T A(G) X + m\| X \|^2\\
			& = \sum_{i \in I} \sum_{uv \in E_i} \left( 2 x_u x_v \right) + m \| X \|^2\\
			& = \sum_{i \in I} \left( \sum_{v \in V(G[E_i])} x_v \right)^2
			- \sum \limits_{v \in V(G)}{m_vx_v^2} + m \| X \|^2\\
			& = \sum_{i \in I} \left( \sum_{v \in V(G[E_i])} x_v \right)^2 +
			\sum \limits_{v \in V(G)} (m-m_v)x_v^2 \ge 0.
		\end{align*}\label{lower_boumd}
		\item If $-m$ is an eigenvalue of $G$, then, from the proof of item \ref{lower_boumd}, equalities \ref{cond1} and
		\ref{cond2} follow. Conversely, if there exists a vector $X \ne \mathbf{0}$ for which \ref{cond1} and \ref{cond2}
		hold, then $X^T A(G) X + m\| X \|^2 = 0$.\\
		Assuming that $\mu$ is the least eigenvalue of $G$, $-m = \frac{X^T A(G) X}{\| X \|^2} \ge \mu$.
		By item~\ref{lower_boumd} we have $\mu \ge -m$ and hence $\mu = - m$. In the positive case, it is immediate
		that $X$ is an eigenvector associated to the eigenvalue $-m$.
		\qedhere
	\end{enumerate}
\end{proof}

From Theorem~\ref{Main_result_1}, it follows that the best lower bound for the least eigenvalue of a graph $G$ is obtained from an ECP, $P$, such that $m_G(P) \le m_G(P')$ for every ECP, $P'$, of $G$. Theorem~\ref{Main_result_1} also provides the spectral lower bound for the content of a graph which appears in \cite{Hoffman1972} and is now stated in the next corollary.

\begin{corollary}
	Let $\mu$ be the least eigenvalue of a graph $G$. Then $-\mu \le C(G)$.
\end{corollary}

\begin{proof}
	If $P$ is a content decomposition of $G$, then, according to Remark~\ref{content_numbe_lower_bound_remark},
	$m_G(P) \le C(G)$. By Theorem~\ref{Main_result_1} $-m_G(P) \le \mu$ and so $-\mu \le C(G)$.
\end{proof}

The following corollaries are also direct consequences of Theorem~\ref{Main_result_1}.

\begin{corollary}
	Let $G$ be a graph of order $n$ and let $X$ be a vector of $\mathbb{R}^n \setminus \{\mathbf{0}\}$. Then $X \in \mathcal{E}_{G}(-m)$ if and only if the conditions \ref{cond1} and \ref{cond2} of Theorem~\ref{Main_result_1}
	hold.
\end{corollary}

\begin{corollary}\label{cor_4}
	Let $P$ be an ECP of a graph $G$. If $-m_G(P)$ is an eigenvalue of $G$, then it is the least eigenvalue of $G$ and for every ECP of $G$, $P'$, $m_G(P') \ge m_G(P)$.
\end{corollary}

Applying Theorem~\ref{Main_result_1}-\ref{lower_boumd} to the graph $G$ with the ECP, $P$, of Example~\ref{ex1}, it
follows that its least eigenvalue $\mu$ is not less than $-m_{G}(P) = -2$. Furthermore, since the necessary and
sufficient conditions \ref{cond1} and \ref{cond2} of Theorem~\ref{Main_result_1} are not fulfilled in the ECP, $P$ (actually, $G$ does not have an ECP fulfilling the necessary and sufficient conditions of Theorem~\ref{Main_result_1}-2), $-m_{G}(P) < \mu$. Additionally, there is no induced subgraph $H$ of $G$ with an ECP, $P'$, such that its least eigenvalue is $\mu' = - m_H(P') \ne -1$. Otherwise, taking into account that the eigenvalues of $H$ interlace the eigenvalues of $G$ \cite[Corollary 1.3.12]{CRS2010}, we obtain
$$
-2 = -m_{G}(P) < \mu \le \mu' = - m_H(P') < -1,
$$
which is a contradiction. It should be noted that $m_H(P')=1$ implies that $H$ is a complete graph.\\

As it is well known, the least eigenvalue of the generalized line graphs (see, e.g. \cite[Def. 1.2.3]{CRS2010}),
which includes the line graphs, is not less that $-2$. However, not every graph with least eigenvalue not less than $-2$ is a generalized line graph. For instance, the least eigenvalue of the Petersen graph is $-2$ and it is not a generalized line graph. A connected graph with least eigenvalue not less than $-2$ which is not a generalized line graph is called exceptional graph \cite[p. 154]{CRS2010}. From Theorem~\ref{Main_result_1}, it follows that a graph $G$ such that $m_{G}(P) = 2$, for some ECP, $P$, is either a generalized line graph or an exceptional graph.\\

Now, as a corollary of Theorem~\ref{Main_result_1}, we state a sharp upper bound on the least eigenvalue of a graph.

\begin{corollary}\label{Main_result_2}
	Let $G$ be a graph with least eigenvalue $\mu$. Assume that $H$ is an induced subgraph of $G$ for which there exists an ECP, $P'$, fulfilling the conditions \ref{cond1} and \ref{cond2} of Theorem~\ref{Main_result_1}. Then
	\begin{enumerate}
		\item $\mu \le - m_H(P')$. \label{condA}
		\item If $G$ admits an ECP, $P$, such that $m_G(P)=m_H(P')$, then $\mu = - m_H(P')$.\label{condB}
	\end{enumerate}
\end{corollary}

\begin{proof}
	If $\mu'$ is the least eigenvalue of $H$, then, by Theorem~\ref{Main_result_1}, $\mu' = - m_H(P')$.
	\begin{enumerate}
		\item Since $H$ is an induced subgraph of $G$, the eigenvalues of $H$ interlace the eigenvalues of $G$. In particular, $\mu \le \mu' = - m_H(P')$.
		\item Assume that $G$ admits an ECP, $P$, such that $m_G(P) = m_H(P')$. Since, by Theorem~\ref{Main_result_1}-\ref{lower_boumd}, $-m_G(P) \le \mu$, it follows that $-m_G(P) \le \mu \le \mu' = - m_H(P')$ and thus $\mu = - m_H(P')$. \qedhere
	\end{enumerate}
\end{proof}
\begin{remark}
	Item \ref{condB} of Corollary~\ref{Main_result_2} states a sufficient condition for $\mu = - m_H(P')$ when $\mu$ is the least eigenvalue of a graph $G$, $H$ is a subgraph of $G$ and $P'$ is an ECP of $H$ fulfilling the conditions \ref{cond1} and \ref{cond2} of Theorem~\ref{Main_result_1}. However, this condition is not a necessary condition for $\mu = - m_H(P')$.
	
	For instance, consider that $G$ is the graph depicted in Figure~\ref{figura_1-B} (which has a subgraph $H$ induced by the vertex subset $\{1, 2, 3, 4\} \subseteq V(G)$ whose ECP, $P'$, fulfills the conditions \ref{cond1} and \ref{cond2} of Theorem~\ref{Main_result_1}). Despite $\mu = -m_H(P')$, there is no ECP, $P$, of $G$ such that $m_G(P) = m_H(P')$.
\end{remark}

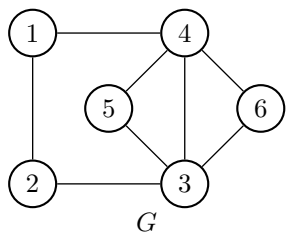
\begin{figure}[ht]
	\centering
	\begin{tikzpicture}
		\begin{scope}[every node/.style={circle,thick,draw}]
			\node (1) at (-1,1) {1};
			\node (2) at (-1,-1) {2};
			\node (3) at (1,-1) {3};
			\node (4) at (1,1) {4};
			\node (5) at (0,0) {5};
			\node (6) at (2,0) {6};
		\end{scope}
		\draw (0.5,-1.5) node {$G$};
		
		\begin{scope}[>={Stealth[black]},
			every edge/.style={draw=black,line width=0.5pt}]
			\path (1) edge node {} (2);
			\path (1) edge node {} (4);
			\path (2) edge node {} (3);
			\path (3) edge node {} (4);
			\path (3) edge node {} (5);
			\path (3) edge node {} (6);
			\path (4) edge node {} (5);
			\path (4) edge node {} (6);
		\end{scope}
		
	\end{tikzpicture}
	
	\caption{A graph $G$ with least eigenvalue $-2$ and a subgraph $H$ induced by the vertex subset $\{1, 2, 3, 4\} \subseteq V(G)$ which admits an ECP, $P'$, fulfilling the hypothesis of Corollary~\ref{Main_result_2}, for which $m_H(P')=2$ but there is no ECP, $P$, of $G$ such that $m_G(P)=m_H(P')$.} \label{figura_1-B}
\end{figure}

From Corollary~\ref{Main_result_2}, we may conclude that the best upper bound for the least eigenvalue of a graph $G$ is obtained from an induced subgraph $H$ having an ECP, $P$, fulfilling the conditions \ref{cond1} and \ref{cond2} of Theorem~\ref{Main_result_1}, such that $m_{H}(P) \ge m_{H'}(P')$ for every induced subgraph $H'$ and every ECP $P'$ of $H'$, fulfilling the conditions \ref{cond1} and \ref{cond2} of Theorem~\ref{Main_result_1}.

\section{An application}\label{sec4}

We start this section with some historical notes about the $n$-Queens' graph which includes its definition. After that, using Theorem~\ref{Main_result_1}, it is proven that the least eigenvalue of the $n$-Queens' graph is equal to $-4$ and its multiplicity is $(n-3)^2$, for every $n \ge 4$.

\subsection{Some historical notes related with the $n$-Queens' graph}\label{sec4.1}
The problem of placing 8 queens on a chessboard such that no two queens attack each other -- i.e. such that there are no queens in the same row, column or diagonal of the chessboard -- was first posed, in 1848, by M. Bezzel, a German chess player \cite{Bezzel}. The German mathematician and physicist Gauss had the knowledge of this problem and found 72 solutions. However, according to \cite{BellStevens}, the first to solve the problem by finding all 92 solutions was F. Nauck in 1850 \cite{Nauck}. As later claimed by Gauss, this number is indeed the total number of solutions. The proof that there is no more solutions was published by E. Pauls in 1874 \cite{Pauls}. The $n$-Queens' problem is a generalization of the above problem, consisting of placing $n$ non attacking queens on $n\times n$ chessboard. In \cite{Pauls} it was also proved that the $n$-Queens' problem has solution for every $n \ge 4$. It is immediate that a solution of the $n$-Queens' problem corresponds to a maximum stable set of the $n$-Queens' graph, $\mathcal{Q}(n)$, below defined, whose stability number is $n$. The problem of determining the number of solutions for an arbitrary $n$, which is equivalent to determining the number of maximum stable sets in $\mathcal{Q}(n)$, remains as an open problem. Recently, in \cite{MS2021} this problem was partially solved.\\

The $n$-Queens' problem has deserved the attention of researchers over the years, belonging to the historical roots of the mathematical approach to domination in graphs which goes back to 1862 \cite{Jaenisch1862}. In the 1970s, the research on the chessboard domination problems was redirected to more general problems of domination in graphs. Since then, this topic has attracted many researchers, turning it into an area of intense research. More recently, in 2017 \cite{GentJeffersenNightingale2017} it was proved that a variant of the $n$-Queens' problem (dating to 1850) called $n$-Queens' completion problem is $\mathbf{NP}$-Complete. In the $n$-Queens' completion problem, assuming that some queens are already placed, the question is to know how to place the rest of the queens, in case such placement be possible. After the publication of \cite{GentJeffersenNightingale2017}, the interest by the $n$-Queens' completion problem has increased. Probably, the motivation is that some researchers believe in the existence of a polynomial-time algorithm to solve this problem (see \cite{Dmitrii2018}). Therefore, if such an algorithm is found, then the problem that asks whether $\mathbf{P}$ is equal to $\mathbf{NP}$ is solved. This problem is one of the seven Millenium Prize Problems stated by the Clay Mathematics Institute which awards one million dollars to anyone who finds a solution to any of them.

The $n$-Queens' graph, $\mathcal{Q}(n)$, associated to the $n \times n$ chessboard $\mathcal{T}_n$, has $n^2$ vertices, each one corresponding to a square of the $n \times n$ chessboard. Two vertices of $\mathcal{Q}(n)$ are adjacent if and only if the corresponding squares in $\mathcal{T}_n$ are in the same row or in the same column or in the same diagonal.

The rows and columns of the chessboard are numbered from the top to the bottom and from the left to the right, respectively. We use the $(i,j) \in [n]^2$ coordinates as labels of the chessboard squares belonging to the $i^{th}$ row and $j^{th}$ column as well as labels of the corresponding vertices in $\mathcal{Q}(n)$. Alternatively, the $n^2$ squares of $\mathcal{T}_n$ and the corresponding $n^2$ vertices in $\mathcal{Q}(n)$ can be labeled by the numbers between $1$ and $n^2$ as it is exemplified in Figure~\ref{chessboard4}~, for the particular case of $\mathcal{T}_4$.

\begin{figure}[H]
	\centering
	\centering
	{	\setchessboard{smallboard}
		\chessboard[
		pgfstyle=
		{[base,at={\pgfpoint{0pt}{-0.4ex}}]text},maxfield=d4,label=false,showmover=false,text= \fontsize{1.2ex}{1.2ex}\bfseries{13},markregions={a1-a1},text= \fontsize{1.2ex}{1.2ex}\bfseries{9},markregions={a2-a2},text= \fontsize{1.2ex}{1.2ex}\bfseries{5},markregions={a3-a3},text= \fontsize{1.2ex}{1.2ex}\bfseries{1},markregions={a4-a4},text= \fontsize{1.2ex}{1.2ex}\bfseries{14},markregions={b1-b1},text= \fontsize{1.2ex}{1.2ex}\bfseries{10},markregions={b2-b2},text= \fontsize{1.2ex}{1.2ex}\bfseries{6},markregions={b3-b3},text= \fontsize{1.2ex}{1.2ex}\bfseries{2},markregions={b4-b4},text= \fontsize{1.2ex}{1.2ex}\bfseries{15},markregions={c1-c1},text= \fontsize{1.2ex}{1.2ex}\bfseries{11},markregions={c2-c2},text= \fontsize{1.2ex}{1.2ex}\bfseries{7},markregions={c3-c3},text= \fontsize{1.2ex}{1.2ex}\bfseries{3},markregions={c4-c4},text= \fontsize{1.2ex}{1.2ex}\bfseries{16},markregions={d1-d1},text= \fontsize{1.2ex}{1.2ex}\bfseries{12},markregions={d2-d2},text= \fontsize{1.2ex}{1.2ex}\bfseries{8},markregions={d3-d3},text= \fontsize{1.2ex}{1.2ex}\bfseries{4},markregions={d4-d4}]}
	\caption{Labeling of $\mathcal{T}_4$.}\label{chessboard4}
\end{figure}

Several combinatorial and spectral properties of the $n$-Queens' graphs are presented in \cite{ICosta2019}.

\subsection{The least eigenvalue of $\mathcal{Q}(n)$, for every $n \ge 4$}\label{sec4.2}

It is useful to start this subsection with the following theorem.

\begin{theorem} \label{nessufcond}
	Let $n\in\mathbb{N}$ such that $n \ge 4$.
	\begin{enumerate}
		\item $-4 \in \sigma(\mathcal{Q}(n))$ if and only if there exists a vector $X \in \mathbb{R}^{n^2}\setminus \{\mathbf{0}\}$
		such that
		\begin{enumerate}
			\item $\sum \limits_{j=1}^n {x_{(k,j)}} = 0$ and $\sum \limits_{i=1}^n {x_{(i,k)}} = 0$, for every $k \in [n]$, \label{nessuf_1}
			\item $\sum \limits_{i+j=k+2}{x_{(i,j)}} = 0$, for every $k \in [2n-3]$, \label{nessuf_2}
			\item $\sum \limits_{i-j=k+1-n} x_{(i,j)} = 0$, for every $k \in [2n-3]$, \label{nessuf_3}
			\item $x_{(1,1)} = x_{(1,n)} = x_{(n,1)} = x_{(n,n)} = 0$. \label{nessuf_4}
		\end{enumerate}
		In the positive case, $X$ is an eigenvector associated with the eigenvalue $-4$.\label{nessuf}
		
		\item If $\mu$ is the least eigenvalue of $\mathcal{Q}(n)$, then $\mu = -4$.\label{least_eigenvalue}
	\end{enumerate}
\end{theorem}

\begin{proof}
	Let us prove each of the items.
	\begin{enumerate}
		\item The proof follows taking into account that the summations \ref{nessuf_1}--\ref{nessuf_3} correspond to the summations \ref{cond1} in Theorem~\ref{Main_result_1}. Here, the cliques obtained from the ECP, $P$, of $\mathcal{Q}(n)$ are the cliques with vertices associated with each of the $n$ columns, $n$ rows, $2n-3$ left to right diagonals and $2n-3$ right to left diagonals. Denoting the vertices of $\mathcal{Q}(n)$ by their coordinates $(i,j)$ in the corresponding chessboard $\mathcal{T}_n$,
		$
		m_{(i,j)}(P) = \left\{\begin{array}{ll}
			3, & \hbox{if } (i,j) \in \{ 1, n\}^2;\\
			4, & \hbox{otherwise}
		\end{array}\right.
		$
		and thus $m_{\mathcal{Q}(n)}(P)=4$. Therefore, the equalities \ref{nessuf_4} correspond to the conditions \ref{cond2} in Theorem~\ref{Main_result_1}.
		\item Consider an induced subgraph $\mathcal{Q}(4)$ of $\mathcal{Q}(n)$. For instance, consider the subgraph induced by the vertices corresponding to the coordinates $(i,j) \in [4]^2$. Let $P'$ be the ECP as defined above in the proof of item \ref{nessuf} (for $n=4$). It is immediate that $m_{\mathcal{Q}(4)}(P')=4$ and the vector $X \in [4]^2$ such that
		$
		X_{(i,j)} = \left\{\begin{array}{rl}
			1, & \hbox{if } (i,j) \in \{(1,2), (2,4), (3,1), (4,3)\};\\
			-1, & \hbox{if } (i,j) \in \{(1,3), (2,1), (3,4), (4,2)\};\\
			0, & \hbox{otherwise}
		\end{array}\right.
		$ fulfills the conditions of item \ref{nessuf}. Therefore, $-m_{\mathcal{Q}(4)}(P')$ is the least eigenvalue of $\mathcal{Q}(4)$. Since $\mathcal{Q}(n)$ admits the ECP, $P$, described above in the proof of item \ref{nessuf} and thus $m_{\mathcal{Q}(n)}(P)=4$, applying Corollary~\ref{Main_result_2}, it follows that $-m_{\mathcal{Q}(4)}(P')$ is the least eigenvalue of $\mathcal{Q}(n)$.
	\end{enumerate}
\end{proof}

As a consequence of Theorem~\ref{nessufcond} we have the following result.

\begin{corollary}\label{non-main_condition1}
	Let $n \geq 4$ and $X \in \mathbb{R}^{n^2} \setminus \{\mathbf{0}\}$. Then $X \in \mathcal{E}_{\mathcal{Q}(n)}(-4)$ if and only if the conditions \ref{nessuf_1}--\ref{nessuf_4} of Theorem~\ref{nessufcond} hold.
\end{corollary}

In what follows we will see that, for $n \ge 4$, $-4$ is an eigenvalue of $\mathcal{Q}(n)$ with multiplicity $(n-3)^2$. From Corollary~\ref{non-main_condition1} we may conclude that the multiplicity of $-4$ as an eigenvalue of $\mathcal{Q}(n)$ coincides with the corank of the coefficient matrix of the system of $6n-2$ linear equations
\ref{nessuf_1}--\ref{nessuf_4}. Therefore, to say that the multiplicity of $-4$ is $(n-3)^2$ is equivalent to say that the rank of the coefficient matrix of the system of $6n$ linear equations \ref{nessuf_1}--\ref{nessuf_4} is $6n - 9$ (since $n^2 - 6n + 9 = (n-3)^2$).

For an easier representation of the vectors, they are displayed over the chessboard. So the $\ell^{th}$ coordinate of a vector $X$ is displayed at the entry of the chessboard corresponding to the vertex $\ell$, i.e. at the entry $(i,j)=(\lceil\frac{\ell}{n}\rceil, \ell+n-n\lceil\frac{\ell}{n}\rceil)$. Then, the $\ell^{th}$ coordinate of $X$ can be denoted by $X_\ell$ or $X_{(i,j)}$.

Before we continue, we need to introduce the family of vectors
$$
\mathcal{F}_n = \{X_n^{(a,b)} \in \mathbb{R}^{n^2} \mid (a,b) \in [n-3]^2 \}
$$
where $X_n^{(a,b)}$ is the vector defined by

\begin{equation}\label{ev_components}
	\big[X_n^{(a,b)}\big]_{(i,j)} = \begin{cases}
		\big[X_4\big]_{(i-a+1,j-b+1)}, & \text{ if } (i,j)\in A \times B;\\
		0,                             & \text{otherwise,}
	\end{cases}
\end{equation}
with $A = \{a,a+1,a+2,a+3\}$, $B = \{b,b+1,b+2,b+3\}$ and $X_4$ is the vector presented in Table~\ref{vector_X4}

\begin{table}[h!]
	\[ 
	\begin{array}{|c|c|c|c|} \hline
		0          & \textbf{1} & \textbf{-1}& 0           \\ \hline
		\textbf{-1}& 0          & 0          & \textbf{1}  \\ \hline
		\textbf{1} & 0          & 0          & \textbf{-1} \\ \hline
		0          &\textbf{-1} & \textbf{1} & 0           \\ \hline
	\end{array}
	\]
	\caption{The vector $X_4$.}\label{vector_X4}
\end{table}

For instance, for $n=5$, $\mathcal{F}_5$ is the family of four vectors depicted in Table~\ref{vetors_11-12-21-22}. 

\begin{table}[h!]\centering
	\begin{tabular}{cc}
		\begin{tabular}{|c|c|c|c|c|} \hline
			0          & \textbf{1} & \textbf{-1}& 0          & 0 \\ \hline
			\textbf{-1}& 0          & 0          & \textbf{1} & 0 \\ \hline
			\textbf{1} & 0          & 0          & \textbf{-1}& 0 \\ \hline
			0          &\textbf{-1} & \textbf{1} & 0          & 0 \\ \hline
			0          & 0          & 0          & 0          & 0 \\ \hline
		\end{tabular}
		&
		\begin{tabular}{|c|c|c|c|c|} \hline
			0 & 0          & \textbf{1} & \textbf{-1}& 0 \\ \hline
			0 & \textbf{-1}& 0          & 0          & \textbf{1} \\ \hline
			0 & \textbf{1} & 0          & 0          & \textbf{-1} \\ \hline
			0 & 0          & \textbf{-1}& \textbf{1} & 0 \\ \hline
			0 & 0          & 0          & 0          & 0 \\ \hline
		\end{tabular}
		\\ \\
		\begin{tabular}{|c|c|c|c|c|}\hline
			0          & 0          & 0          & 0          & 0 \\ \hline
			0          & \textbf{1} & \textbf{-1}& 0          & 0\\ \hline
			\textbf{-1}& 0          & 0          & \textbf{1} & 0 \\ \hline
			\textbf{1} & 0          & 0          & \textbf{-1}& 0 \\ \hline
			0          & \textbf{-1}& \textbf{1} & 0          & 0 \\ \hline
		\end{tabular}&
		\begin{tabular}{|c|c|c|c|c|}\hline
			0 & 0          & 0          & 0          & 0  \\ \hline
			0 & 0          & \textbf{1} & \textbf{-1}& 0 \\ \hline
			0 & \textbf{-1}& 0          & 0          & \textbf{1} \\ \hline
			0 & \textbf{1} & 0          & 0          & \textbf{-1} \\ \hline
			0 & 0          & \textbf{-1}& \textbf{1} & 0 \\ \hline
		\end{tabular}
	\end{tabular}
	\caption{The vectors $X_5^{(1,1)}$, $X_5^{(1,2)}$, $X_5^{(2,1)}$, and $X_5^{(2,2)}$.} \label{vetors_11-12-21-22}
\end{table}

\begin{theorem}\label{th-4eigenvalue}
	$-4$ is an eigenvalue of $\mathcal{Q}(n)$ with multiplicity $(n-3)^2$ and $\mathcal{F}_n$ is a basis for
	$\mathcal{E}_{\mathcal{Q}(n)} (-4)$.
\end{theorem}

\begin{proof}
	First, note that every element of $\mathcal{F}_n$ belongs to $\mathcal{E}_{\mathcal{Q}(n)} (-4)$. Indeed, if $X = \left( x_{(i,j)} \right) \in \mathcal{F}_n$, then the conditions \ref{nessuf_1}--\ref{nessuf_4} of Theorem~\ref{nessufcond} hold and hence by Corollary~\ref{non-main_condition1} $X \in \mathcal{E}_{\mathcal{Q}(n)} (-4)$.
	
	Second, $\mathcal{F}_n$ is linearly independent and so $\dim \mathcal{E}_{\mathcal{Q}(n)} (-4) \geq (n-3)^2$. For otherwise there would be scalars $\alpha_{1,1}, \ldots, \alpha_{n-3,n-3} \in \mathbb{R}$, not all equal to zero, such that
	\begin{equation} \label{comblin}
		\alpha_{1,1} X_n^{(1,1)} + \cdots + \alpha_{n-3,n-3} X_n^{(n-3,n-3)} = \mathbf{0}.
	\end{equation}
	
	Let $(n-3)(a-1)+b$ be the smallest integer such that $\alpha_{a,b} \neq 0$. Since by \eqref{ev_components} $\big[X_n^{(a,b)}\big]_{(a,b+1)} = \big[X_4\big]_{(1,2)}=1$, the entry $(a,b+1)$ of $\alpha_{a,b}\big[X_n^{(a,b)}\big]$ is $\alpha_{a,b}$. Consider any other vector $\big[X_n^{(a',b')}\big]$ such that $(n-3)(a'-1)+b'>(n-3)(a-1)+b$ which implies (i) $a'>a$ or (ii) $a'=a$ and $b'>b$. Denoting $A'=\{a', \dots, a'+3\}$ and $B'=\{b', \dots, b'+3\},$ taking in to account \eqref{ev_components}, we may conclude the following.
	\begin{itemize}
		\item[(i)] $a'>a$ implies $(a,b+1) \not \in A' \times B'$ and thus $\big[X_n^{(a',b')}\big]_{(a,b+1)} =0$.
		\item[(ii)] For $a'=a$ and $b'>b+1$ the conclusion is the same as above. Assuming $a'=a$ and $b'=b+1$ it follows that
		$\big[X_n^{(a',b')}\big]_{(a,b+1)}=\big[X_4\big]_{(1,1)}=0$.
	\end{itemize}
	Therefore, entry $(a,b+1)$ of the left-hand side of (\ref{comblin}) is $\alpha_{a,b} \neq 0$ while the same entry on the right-hand side
	of (\ref{comblin}) is 0, which is a contradiction.
	
	Finally, we show that $\dim (\mathcal{E}_{\mathcal{Q}(n)} (-4)) \le (n-3)^2$ by showing that every element of the
	subspace generated by $\mathcal{F}_n$ is completely determined by entries $x_{(i,j+1)}$ such that $(i,j) \in [n-3]^2$.
	
	Let $S \subseteq [n]^2$ be the set of indexes $(p,q) \in [n]^2$ such that the entry $x_{(p,q)}$ of
	$X \in \mathcal{E}_{\mathcal{Q}(n)} (-4)$ is completely determined by the entries $x_{(i,j+1)}$, with
	$(i,j) \in [n-3]^2$. Clearly, $[n-3] \times ([n-2] \setminus \{ 1 \}) \subseteq S$. Since $x_{(1,1)} = x_{(n,1)} = 0$, it follows that
	\begin{eqnarray*}
		x_{(i,1)} &=& - \sum_{k=2}^i x_{(i+1-k,k)}, \ \text{for every $2 \leq i \leq n-2$,}\\
		x_{(n-1,1)} &=& - \sum_{k=2}^{n-2} x_{(k,1)} \\
		&=& \hphantom{+} x_{(1,2)} \\
		& & + x_{(2,2)} + x_{(1,3)} \\
		& & \vdots \\
		& & + x_{(n-3,2)} + \cdots + x_{(2,n-3)} + x_{(1,n-2)} \\
		&=& \sum_{\substack{i,j \geq 1 \\ i+j \leq n-2}} x_{(i,j+1)}
	\end{eqnarray*}
	and then $[n] \times \{ 1 \} \subseteq S$. Additionally, since $x_{(1,n)} = x_{(n,n)} = 0$ it follows that
	\begin{align*}
		& x_{(i,n-1)} = - \sum_{j=1}^{n-2} x_{(i,j)} - x_{(i,n)}, \ \text{for every $1 \leq i \leq n-3$,} \\
		& x_{(i+1,n)} = - \sum_{k=1}^{i} x_{(k,n-1-i+k)}, \ \text{for every $1 \leq i \leq n-3$,} \\
		& x_{(n-1,n)} = - \sum_{i=2}^{n-2} x_{(i,n)}, \quad  x_{(n-2,n-1)} = - \sum_{k=1}^{n-3} x_{(k,k+1)} - x_{(n-1,n)}, \\
		& x_{(n,n-1)} = - x_{(n-1,n)}, \quad x_{(n-1,n-1)} = - \sum_{i=2}^{n-2} x_{(i,n-1)} - x_{n,n-1}
	\end{align*}
	and thus $[n] \times \{ n-1, n \} \subseteq S$. Finally, since for every $2 \le j \le n-2$
	\begin{align*}
		& x_{(n,j)} = - \sum_{k=j}^{n-1} x_{(k,n+j-k)}, \\
		& x_{(n-2,j)} = - \sum_{k=1}^{j-1} x_{(n-2-j+k,k)} - x_{(n-1,j+1)} - x_{(n,j+2)},\\
		& x_{(n-1,j)} = - \sum_{i=1}^{n-2} x_{(i,j)} - x_{(n,j)},
	\end{align*}
	and consequently $\{ n-2, n-1, n \} \times ([n-2] \setminus \{ 1 \} ) \subseteq S$.
\end{proof}

\section{Some conclusions and remarks}\label{sec5}

Consider the graph $H$ and the ECP, $P'$, depicted in Figure~\ref{figura_3}. It is immediate that $m_H(P')=2$, that the vector $X \in \{-1, 0, 1\}^5$, whose entries are displayed on the right, fulfills the necessary and sufficient conditions \ref{cond1}--\ref{cond2} of Theorem~\ref{Main_result_1} and thus its least eigenvalue is $-m_H(P') = -2$.

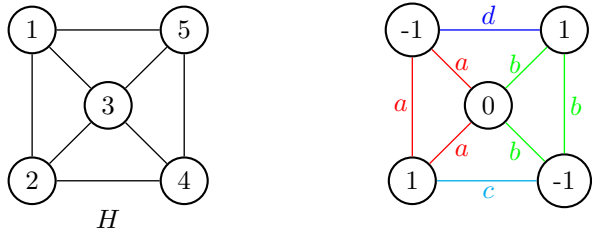
\begin{figure}[ht]
	\centering
	\begin{tikzpicture}
		\begin{scope}[every node/.style={circle,thick,draw}]
			\node (1) at (-1,1) {1};
			\node (2) at (-1,-1) {2};
			\node (3) at (0,0) {3};
			\node (4) at (1,-1) {4};
			\node (5) at (1,1) {5};
			\node (1') at (4,1) {-1};
			\node (2') at (4,-1) {1};
			\node (3') at (5,0) {0};
			\node (4') at (6,-1) {-1};
			\node (5') at (6,1) {1};
		\end{scope}
		\draw (0,-1.5) node {$H$};
		
		\begin{scope}[>={Stealth[black]},
			every edge/.style={draw=black,line width=0.5pt}]
			\path (1) edge node {} (2);
			\path (1) edge node {} (3);
			\path (1) edge node {} (5);
			\path (2) edge node {} (3);
			\path (2) edge node {} (4);
			\path (3) edge node {} (4);
			\path (3) edge node {} (5);
			\path (4) edge node {} (5);
			\path (1') edge[color=red] node {} (2');
			\path (1') edge[color=red] node {} (3');
			\path (1') edge[color=blue] node {} (5');
			\path (2') edge[color=red] node {} (3');
			\path (2') edge[color=cyan] node {} (4');
			\path (3') edge[color=green] node {} (4');
			\path (3') edge[color=green] node {} (5');
			\path (4') edge[color=green] node {} (5');
		\end{scope}
		\draw (4.65,0.55) node {\color{red}$a$};
		\draw (3.85,0) node {\color{red}$a$};
		\draw (4.65,-0.6) node {\color{red}$a$};
		\draw (5,-1.15) node {\color{cyan}$c$};
		\draw (5,1.2) node {\color{blue}$d$};
		\draw (5.35,0.55) node {\color{green}$b$};
		\draw (6.15,0) node {\color{green}$b$};
		\draw (5.35,-0.6) node {\color{green}$b$};
		
	\end{tikzpicture}
	\caption{A graph with a content decomposition $P'$, where the edges with the same color, among the colors {\color{red}$a$}, {\color{green}$b$}, {\color{cyan} $c$}, and {\color{blue} $d$}, belong to the same part. The labels of the vertices on the right are the entries of the vector $X$ considered in Theorem~\ref{Main_result_1}.}\label{figura_3}
\end{figure}

The graph $H$ in Figure~\ref{figura_3} is an induced subgraph of the graph $G_2$ in Figure~\ref{figura_4}. By Corollary~\ref{Main_result_2}, if $\mu$ is the least eigenvalue of $G_2$, then $\mu \le -m_H(P') = -2$.

Now consider a graph $H$ with an ECP, $P$. Assume that for some $k \ge m_{H}(P)$, $G_k \in \mathcal{G}(H)$ admits an ECP, $P_k$, and exists a vector $X \in \{-1,0,1\}^{|V(G_k)|}$ indexed by the vertices of $G_k$ fulfilling the necessary and sufficient conditions of Theorem~\ref{Main_result_1} and thus the least eigenvalue of $G_k$ is $-k$. Then, choosing the identity as the permutation $\pi_k$ and defining a vector $Y \in \{-1, 0, 1\}^{|V(G_{k+1})|}$ indexed by the vertices of $G_{k+1}$ such that $Y_{v}= X_{v}$ for every $v \in V(G_k)$ and $Y_{v'} = - X_{v}$ for every $v' \in V(G'_k)$, it is immediate that $Y$ fulfills the necessary and sufficient conditions of Theorem~\ref{Main_result_1} and thus the least eigenvalue of $G_{k+1}$ is $-(k+1)$. Fixing $X$ and $Y$ as described above, if $G_{k+1}$ is obtained from $G_k$ replacing the identity permutation $\pi_k$ by a more general permutation $\pi'_k$ such that
\begin{equation}
	\forall v \in V(G_k) \qquad \pi'_k(v) = u' \in V(G'_k), \;\; \hbox{ if } Y_{u'} = - X_v,
	\label{alternative_permutation}
\end{equation}
it follows that the least eigenvalue of $G_{k+1}$ remains equal to $-(k+1)$.\\

Let us consider each of the two cases.

\begin{itemize}
	\item[(i)] When $\pi_k$ is the identity permutation we may conclude that $\sigma(G_{k+1}) = \sigma(G_k) \pm 1$, where $\sigma(G_k) \pm 1 = \{\lambda \pm 1 \mid \lambda \in \sigma(G_k) \}$ with possible repetitions. Indeed, assuming that $\lambda$ is an eigenvalue of $G_k$ and $X$ is an eigenvector associated with $\lambda$, then
	\begin{eqnarray*}
		A(G_{k+1})\left(\begin{array}{r}
			X \\
			\pm X \\
		\end{array}\right)&=&\left(\begin{array}{cc}
			A(G_k) &   I \\
			I   & A(G_k) \\
		\end{array}\right)\left(\begin{array}{r}
			X \\
			\pm X \\
		\end{array}\right)=(\lambda \pm1)\left(\begin{array}{r}
			X\\
			\pm X
		\end{array}\right),
	\end{eqnarray*}
	where $I$ is the identity matrix of order $|V(G_k)|$. Therefore, $\sigma(G_{k+1})=\sigma(G_k) \pm 1$.
	In particular, considering the graph $G_2$ as the graph $H$ depicted in Figure~\ref{figura_3}, since
	$\sigma(G_2) = \{-2, 1-\sqrt{5}, 0, 0, 1+\sqrt{5}\}$ it follows that
	\[
	\sigma(G_3) = \{-3, -\sqrt{5}, -1, -1, \sqrt{5}, -1, 2-\sqrt{5}, 1, 1, 2+\sqrt{5}\}.
	\]
	\item[(ii)] When $\pi_k$ is replaced by a more general permutation $\pi'_k$, as defined in \eqref{alternative_permutation}, the equality
	$\sigma(G_{k+1}) = \sigma(G_k) \pm 1$ may not be true. For instance, let $G_3 \in~\mathcal{G}(H)$ be the graph depicted in Figure~\ref{figura_4} which is obtained from $G_2$ as above described,
	using the ECP $P=\{\{12, 13, 35\}, \{24\}, \{34, 35, 45\}, \{36\},$ $\{15,16,56\}\}\}$ of $G_2$, the vector $X^T=(1,-1,0,1,-1,0)$ and the permutation $\pi'_2=[156423]$. Then it follows that
	$\sigma(G_3) \ne \sigma(G_2) \pm 1$. However, $-k$ and $-k-1$ are the least eigenvalues of $G_k$ and $G_{k+1}$, respectively, and this property is invariant to the permutations $\pi_k \in \{\pi'_k \mid \pi'_k \text{ is defined by \eqref{alternative_permutation}}\}$.
	
	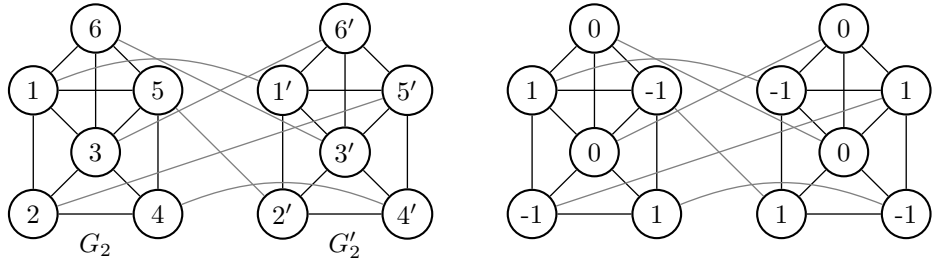
\begin{figure}[H]
		
		\centering
		\begin{tikzpicture}[scale=0.82]
			\begin{scope}[every node/.style={circle,thick,draw,align=center,inner sep=0pt,text width=6mm}]
				\node (1) at (-1,1) {1};
				\node (2) at (-1,-1) {2};
				\node (3) at (0,0) {3};
				\node (4) at (1,-1) {4};
				\node (5) at (1,1) {5};
				\node (6) at (0,2) {6};
				\node (1') at (3,1) {$1'$};
				\node (2') at (3,-1) {$2'$};
				\node (3') at (4,0) {$3'$};
				\node (4') at (5,-1) {$4'$};
				\node (5') at (5,1) {$5'$};
				\node (6') at (4,2) {$6'$};
				\node (v1) at (7,1) {1};
				\node (v2) at (7,-1) {-1};
				\node (v3) at (8,0) {0};
				\node (v4) at (9,-1) {1};
				\node (v5) at (9,1) {-1};
				\node (v6) at (8,2) { 0};
				\node (v1') at (11,1) {-1};
				\node (v2') at (11,-1) {1};
				\node (v3') at (12,0) {0};
				\node (v4') at (13,-1) {-1};
				\node (v5') at (13,1) {1};
				\node (v6') at (12,2) {0};
			\end{scope}
			\draw (0,-1.5) node {$G_2$};
			\draw (4,-1.5) node {$G_2'$};
			
			\begin{scope}[>={Stealth[black]},
				every edge/.style={draw=black,line width=0.5pt}]
				\path (1) edge node {} (2);
				\path (1) edge node {} (3);
				\path (1) edge node {} (5);
				\path (1) edge node {} (6);
				\path (2) edge node {} (3);
				\path (2) edge node {} (4);
				\path (3) edge node {} (6);
				\path (4) edge node {} (3);
				\path (4) edge node {} (5);
				\path (5) edge node {} (3);
				\path (5) edge node {} (6);
				\path (1') edge node {} (2');
				\path (1') edge node {} (3');
				\path (1') edge node {} (5');
				\path (1') edge node {} (6');
				\path (2') edge node {} (3');
				\path (2') edge node {} (4');
				\path (3') edge node {} (6');
				\path (4') edge node {} (3');
				\path (4') edge node {} (5');
				\path (5') edge node {} (3');
				\path (5') edge node {} (6');
				\path (1) edge[bend left=22,color=gray] node {} (1');
				\path (2) edge[color=gray] node {} (5');
				\path (4) edge[bend left=22,color=gray] node {} (4');
				\path (5) edge[color=gray] node {} (2');
				\path (3) edge[color=gray] node {} (6');
				\path (6) edge[color=gray] node {} (3');
				
				\path (v1) edge node {} (v2);
				\path (v1) edge node {} (v3);
				\path (v1) edge node {} (v5);
				\path (v1) edge node {} (v6);
				\path (v2) edge node {} (v3);
				\path (v2) edge node {} (v4);
				\path (v3) edge node {} (v6);
				\path (v4) edge node {} (v3);
				\path (v4) edge node {} (v5);
				\path (v5) edge node {} (v3);
				\path (v5) edge node {} (v6);
				\path (v1') edge node {} (v2');
				\path (v1') edge node {} (v3');
				\path (v1') edge node {} (v5');
				\path (v1') edge node {} (v6');
				\path (v2') edge node {} (v3');
				\path (v2') edge node {} (v4');
				\path (v3') edge node {} (v6');
				\path (v4') edge node {} (v3');
				\path (v4') edge node {} (v5');
				\path (v5') edge node {} (v3');
				\path (v5') edge node {} (v6');
				\path (v1) edge[bend left=22,color=gray] node {} (v1');
				\path (v2) edge[color=gray] node {} (v5');
				\path (v4) edge[bend left=22,color=gray] node {} (v4');
				\path (v5) edge[color=gray] node {} (v2');
				\path (v3) edge[color=gray] node {} (v6');
				\path (v6) edge[color=gray] node {} (v3');
			\end{scope}
		\end{tikzpicture}
		\caption{Graph $G_3$ obtained from $G_2$ using the vector $X^T=(1,-1,0,1,-1,0)$ and  the permutation $\pi'_2=[156423]$. The labels of the vertices on the right are the entries of the vector $Y$ obtained from $X$ as above described.}\label{figura_4}
	\end{figure}
\end{itemize}
In any case, if the least eigenvalue of $G_k$ for  $k=m_H(P)$ is $\mu=-m_H(P)$, then, by induction on $k$, the least
eigenvalue of $G_k$ is $-k$, for every $k \ge m_H(P)$. \\

Finally, it is immediate that if $H$ is an integral graph, i.e., a graph whose eigenvalues are all integers and the family $\mathcal{G}(H)$ is produced using the identity permutation, then all the graphs of $\mathcal{G}(H)$ are integral graphs. \\

\medskip\textbf{Acknowledgments.}
The authors thank Sebastian Cioab\u{a} for drawing our attention to the Corollary 3.1 that appears in \cite{CioabaElzingaGregory2020} and Edwin van Dam for his encouraging comments. This work is supported by the Center for Research and Development in Mathematics and Applications (CIDMA) which is financed by national funds through Funda\c{c}\~{a}o para a Ci\^{e}ncia e a Tecnologia (FCT), within project UIDB/04106/2020. I.S.C. also thanks the support of FCT via PhD Scholarship PD/BD/150538/2019.

\end{document}